\documentclass{elsarticle}
\usepackage{amssymb,amsmath}
\usepackage{amsthm}
\newtheorem{Theorem}{Theorem}[section]

\newtheorem{Definition}[Theorem]{Definition}

\newtheorem{Lemma}[Theorem]{Lemma}

\newtheorem{Example}[Theorem]{Example}

\newtheorem{Remark}[Theorem]{Remark}

\newtheorem{Proposition}[Theorem]{Proposition}

\numberwithin{equation}{section}

\newcommand{\DS}{\displaystyle}

\newcommand{\di}{\mathrm{diam}\,}

\begin{document}

\title{Some fixed points results for extended nonlinear contractions defined on  $d-CS$ spaces}

\author{I. D. Aran{\dj}elovi\'c}
\address{University of Belgrade, Faculty of Mechanical
Engineering, Kraljice Marije 16, 11000, Beograd, Serbia}
\emailauthor{iarandjelovic@mas.bg.ac.rs}{I. D. Aran{\dj}elovi\'c}

\author{Z. D. Mitrovi\'c\corref{cor1}}
\address{University of Banja Luka, Faculty of Electrical Engineering, Patre 5, 78000, Banja Luka, Bosnia and Herzegovina}

\emailauthor{zoran.mitrovic@etf.unibl.org}{Z. D. Mitrovi\'c\corref{cor1}}

\cortext[cor1]{Corresponding author}

\def\Dj{\hbox{D\kern-.73em\raise.30ex\hbox{-} \raise-.30ex\hbox{}}}
 \def\dj{\hbox{d\kern-.33em\raise.80ex\hbox{-} \raise-.80ex\hbox{\kern-.40em}}}
\def\dD{\hbox{d\kern-.33em\raise.75ex\hbox{-}\raise-.75ex\hbox{}}}

\begin{abstract} In this paper we shall introduced the class 
of $d-CS$ spaces and so we shall obtained topological approach to  large Kasahara spaces. This class include complete symmetric spaces, complete quasi $b$-metric spaces and complete $b$-spaces, but not include $d^*$-complete topological spaces and $d$-complete topological spaces. Further will be presented fixed theorem for nonlinear extended-contraction defined on this spaces, which generalize earlier results obtaned by Browder, Matkowski, Jachymski, Matkowski and \'Swiatkowski, Aran{\dj}elovi\'{c} and Ke\v{c}ki\'{c}  and Alshehri, Aran{\dj}elovi\'c and Shahzad. Also, we obtain the fixed point theorem for nonlinear extended-contraction defined on quasi-metric spaces which extended recent result's of Pasicki, obtained for pseudo-metric spaces.

 
\end{abstract}


\maketitle

\section{Introduction}

Exetended concepts of continuity and convergence, in paper \cite{MF1} M.\ Fr\'{e}\-chet  introduced the classes of metric spaces and $E$-spaces (in modern terminology so called symmetric spaces). Later in \cite{MF2}, he introduced the notion of $L$-spaces and presented one different approach to this concepts. In this class of spaces, which need not be topological spaces, the class of convergent sequences is axiomatic introduced. S. Kasahara \cite{SK} considered fixed point results on  $d$-complete $L$ spaces, also called Kasahara spaces (see \cite{IR}). T.\ Hicks \cite{TH} defined the notion of  $d$-complete topological spaces and so obtained topological approach to  these spaces, which is extended in \cite{MAMAS}  in form  of  $d^*$-complete topological spaces. Large Kasahara spaces was introduced by I.~Rus \cite{IR}. Concept of quasi-metric space was presented by W. Wilson \cite{WW}.\\\indent
In this paper we shall introduced the class 
of $d - CS$ spaces and so we shall obtained topological approach to  large Kasahara spaces. This class include complete symmetric spaces, complete quasi $b$-metric spaces and complete $b$-spaces, but not include $d^*$-complete topological spaces and $d$-complete topological spaces. Further will be presented fixed theorem for nonlinear extended-contraction defined on this spaces, which generalize earlier results obtaned by Browder \cite{FB}, Matkowski \cite{JM}, Jachymski, Matkowski and \'Swiatkowski \cite{JMS}, Aran{\dj}elovi\'{c} and Ke\v{c}ki\'{c} \cite{AK} and Alshehri, Aran{\dj}elovi\'c and Shahzad \cite{AAS}. Also, we obtain the fixed point theorem for nonlinear extended-contraction defined on quasi-metric spaces which extended recent result's of Pasicki \cite{LP1,LP2}, obtained for pseudo - metric spaces.
\section{Preliminary Notes}

\subsection{Fixed point theory}
Let $X$ be a nonempty set and $f:X\rightarrow X$ be an arbitrary mapping. An arbitrary element $x\in X$ is a fixed point for
 $f$ if $x=f(x)$. For $\vartheta_0\in X$, we say that a sequence $(\vartheta_n)$ defined by
 $\vartheta_n=f^n(\vartheta_0)$ is a sequence of Picard iterates of $f$ at point $\vartheta_0$ or that $(\vartheta_n)$
is the orbit of $f$ at point $\vartheta_0$.\\\indent
Let  $d: X\times X \rightarrow [0,+\infty)$  be a mapping, then:

\noindent (1) $f$ is contraction if there exists  $\alpha\in [0,1)$ such that 
\begin{equation}
d(f(x),f(y))\le \alpha d(x,y),
\end{equation}
for all $x,y\in X$;\\
(2) $f$ is nonlinear contraction if there exists  function $\varphi:[0,+\infty) \rightarrow [0,+\infty)$ such that $\varphi(r) < r$ for any $r>0$  and  
\begin{equation}
d(f(x),f(y))\le \varphi(d(x,y)), 
\end{equation}
 for all $x,y\in X$;\\
(3) $f$ is nonlinear extended-contraction if there exists  function $\varphi:[0,+\infty) \rightarrow [0,+\infty)$ such that $\varphi(r) < r$ for any $r>0$ and 
\begin{equation}
d(f(x),f(y))\le \max\{\varphi (d(x,y)),\varphi (d(x,f(x))),\varphi(d(y,f(y)))\}, 
\end{equation}
 for all $x,y\in X$.
 
In 1968 first fixed point theorems for nonlinear contractions on complete metric space were obtained by F.\ Browder \cite{FB} (his statemet has assumptions that considered metric space is bounded and that comparison function $\varphi$ is monotone non - decreasing and right continuous), R.\ M.\ Bianchini and M.\ Grandolfi \cite{BG} (their statemet has assumptions that $\varphi$ is monotone non - decreasing and that $\sum_{n=1}^{\infty}\varphi^n(t)<\infty$ for each $t>0$), M.\ Furi \cite{MF} and A.~Zitarosa \cite{AZ} (it has assumption that $\varphi$ is monotone non - decreasing, such that for each $r>0$ $\DS{\lim_{n\rightarrow\infty}\varphi^n(r) = 0}$, and that orbits of $f$ are bounded).

\begin{Remark} In A.\ Zitarosa \cite{AZ} theorem assumption $\varphi (t)<t$ was omitted beacuse it follows from two other assumptions.
\end{Remark}

\begin{Remark} Results of R.\ M.\ Bianchini and M.\ Grandolfi \cite{BG}, M.\ Furi \cite{MF} and Browder  \cite{FB} are special cases of the Theorem of A.\ Zitarosa \cite{AZ}.
\end{Remark}

\bigskip In 1969 D.\ W.\ Boyd, and J.\ S.\ W.\ Wong \cite{BW} presented two new fixed point results for nonlinear contractions under assumptions:

\medskip\noindent a) $\varphi$ is upper semi continuous from the right (i.\ e.\ for every $r\ge 0$ $\DS{\overline{\lim_{t\rightarrow r+}}} \varphi(t)\le\varphi(r)$, or

\medskip\noindent b) $\DS{\overline{\lim_{t\rightarrow r+}}} \varphi(t)< r$ for each $r\ge 0$.

\medskip\noindent This conditions are eqivalent, but asssumptions of Boyd-Wong and Zitarosa are uncomparable. 

\begin{Remark} Result of Browder \cite{FB} follows from the  both theorems of Boyd and Wong \cite{BW}.
\end{Remark}

In 1975 J.\ Matkowski \cite{JM} that condition "orbits of $f$ are bounded" in Theorem of Zitarosa can be omitted. This result extended Theorem of Zitarosa but it is uncomparable with both theorems of Boyd and Wong.

After a long time, in 2016 L.\ Pasicki \cite{LP1} introduced new nonlinear contractive condition which include assumptions of  J.\ Matkowski \cite{JM} and Boyd and Wong \cite{BW}. Result of Pasicki was formulated for mappings defined on pseudo-metric spaces. His theorem has assumption that for each $t > 0$ there exists $\varepsilon > 0$ such that 
$$\varphi (s) \le t\enskip \mbox{ for any } s\in (t,t+\varepsilon).$$

Fixed point results for nonlinear contractions defined on semi-metric spaces was obtained by J.\ Jachymski, J.\ Matkowski and T.\ \'Swiatkowski \cite{JMS} and for nonlinear contractions defined on symmetric spaces was presented by I.\ D.\ Aran{\dj}elovi\'{c} and D.\ J.\ Ke\v{c}ki\'{c} \cite{AK}.

First fixed point result on nonlinear extended-contractions was presented by M.\ Maiti, J.\ Achari, T.\ K.\ Pal \cite{MAP}.
It proved that any extended contraction $f$ defined on complete metric space $(X,d)$, where $\varphi$ satisfies condition $a)$ of Boyd-Wong's theorem and function $A:X\rightarrow \mathbb{R}$ defined by $A(x)=d(x,f(x))$ is lower semicontinuous,
has an unique fixed point which is unique limit of all sequences of its Picard iterations. In 2016 L.\ Paciski \cite{LP1} generalize this result to pseudo-metric space. Corresponding result for nonlinear extended-contractions defined on symmetric spaces was presented by S.\ Alshehri, I.\ Aran{\dj}elovi\'c and N. Shahzad \cite{AAS}. Fixed point theorem for nonlinear extended-contraction defined on cone metric spaces was proved in \cite{AKR}. In 2018 L.\ Paciski \cite{LP2} obtained corresponding result to quasi pseudo-metric space.

\medskip The first part of  next statement was formulated and proved by D.~Adam\-ovi\'c \cite{DA}.
Its second part was presented in \cite{AK}.

\begin{Lemma} \label{L2.1} (Aran{\dj}elovi\'c-Ke\v cki\'c \cite{AK}) Let $X$ be the nonempty set and the mapping $f:X \rightarrow X$. Let $l$ be a positive integer such that $f^{l}$ possesses a unique fixed point, say $u_*$. Then\ $u_*$ is the unique fixed point of $f$. Also,  if $X$ is a topological space and any sequence of Picard
iterates defined by $f^l$ is convergent  to $u_*$, then the sequence of Picard iterates defined by $f$ is
convergent to $u_*$.
\end{Lemma}

\subsection{Topological definitions}
In this subsection we give some definitions.
\begin{Definition}\label{D2.1} Let $X$ be a nonempty set, 
$d: X\times X\rightarrow [0,+\infty)$ and $s\in\mathbb{R}$. We define the following five properties:\\
(A0)  $d (x,y) = 0$ implies  $x = y$;\\
(A1)  $d (x,y) = 0$ if and only if  $x= y$;\\
(A2)  $d(x,y) = d (y,x)$;\\
(A3)  $d (x,y)\le d(x,z) + d (z,y)$;\\
If for any $x,y,z\in X$, $(X,d)$ satisfies:\\
(A0), (A2) and (A3), then  $(X,d)$  is pseudo-metric space;\\
(A1), (A2) and (A3), then  $(X,d)$ is metric space;\\
(A1) and (A2), then  $(X,d)$  is symmetric space;\\
(A0) and (A3), then  $(X,d)$  is pseudo quasi-metric space;\\
(A1) and (A3), then  $(X,d)$  is quasi-metric space.
\end{Definition}

We shall used classical term  pseudo-metric spaces, as in famous textbooks \cite{JK} or \cite{SL}). In the last 20 years this notion was reintroduced by P.\  Hitzler and A.\ K.\ Seda \cite{HS} (they used term dislocated metric spaces) and A.\ Amini - Harandi \cite{AAH} (he used term metric like spaces).
\begin{Example}Let $X=\mathbb{R}$ and $d:X\times X\rightarrow [0, +\infty)$ defined by
\begin{equation*}
d(x, y)=|y-x|+\frac{y-x}{2},
\end{equation*}
for $x, y\in X$. Then $(X, d)$ is a quasi-metric space. Note that (A2) does not hold for the mapping $d$.
\end{Example}
Let $(X,d)$ be a quasi-metric spaces and $r_n$ be a sequence of nonnegative real numbers such that $r_{n+1} \le r_n$ and $\lim r_n = 0$. A quasi-metric space is a topological space with $\{B(x,r_n)\}$, as a
base of neighborhood filter of the point $x$ where 
$$B(x,r_n)= \{y\in X:d(y,x)<r_n\}.$$ 

\begin{Definition}\label{D2.2} 
Let $(X,d)$ be a quasi-metric space. A sequence $( x_{n})\subseteq X$ \ is said to be left Cauchy sequence if for given $\varepsilon >0$ there is
$N\in \mathbb{N}$ such that $d(x_{n},x_{m})<\varepsilon $, for all
$m>n\geq N.$ Then $(X,d)$ is complete if and only if every
left Cauchy sequence converges to some $x\in X$. \end{Definition}

Let $(X,d)$ and $(Y,d)$ be two quasi-metric spaces. A mapping $f:X\rightarrow Y$ is sequentially continuous if for each sequence
 $(\vartheta_n)\subseteq X$ from $\lim d(\vartheta_{n},p)=0$, it follows that $\lim d(f(\vartheta_n,f(p))=0$.

\subsection{Comparison Functions}

Let $\varphi:[0,+\infty) \rightarrow [0,+\infty)$ be mapping which satisfies:
$\varphi (0) = 0$ and $\varphi(r) < r$ for any $r>0$. Then $\varphi$ is comparison function.

\begin{Proposition}\label{P2.1} Let $\varphi:[0,+\infty) \rightarrow [0,+\infty)$ be monotone non-decreasing comparison function, such that for each $r>0$ $\DS{\lim_{n\rightarrow\infty}\varphi^n(r) = 0}$. If there exists $r>0$, which satisfies
$$\overline{\lim_{t\rightarrow r+}}\varphi (t)=r$$ then there exists $\varepsilon > 0$ such that 
$$\varphi (t) = r\enskip \mbox{ for any } t\in (r,r+\varepsilon).$$
\end{Proposition}

\begin{proof} From 
$$\overline{\lim_{t\rightarrow r+}}\varphi (t)=r,$$
it follows that
$$\lim_{t\rightarrow r+}\varphi (t)=r,$$
because $\varphi$ is monotone non - decreasing. If for each $\varepsilon >0$ there exists $t\in (r,r+\varepsilon)$ such that
$\varphi (t)>r$, then we obtain that $\varphi:(r,+\infty)\rightarrow (r,+\infty)$ which is contradictin with $\DS{\lim_{n\rightarrow\infty}\varphi^n(r+\varepsilon) = 0}$.
\end{proof}

Proposition \ref{P2.1} implies that fixed point theorem of Matkowski \cite{JM} is included in result of Pasicki \cite{LP1}. 

\begin{Proposition}\label{P2.2} Let $\varphi:[0,+\infty) \rightarrow [0,+\infty)$ be comparison function, such that one of the following conditions be
satisfied:\\
$\alpha$) $\DS{\overline{\lim_{t\rightarrow s+}}\varphi (t)< s}$ for each $s>0$;\\
$\beta$) for any $s>0$ there exists $\varepsilon >0$ such that $\varphi (t) \le s\enskip \text{ for any } t\in (s,s+\varepsilon)$,\\
then  $\DS{\lim_{n\rightarrow+\infty}\varphi^n(s) = 0}$ and for each $s>0$. 
\end{Proposition}

\begin{proof}

Let $\alpha$) be satisfied. Suppose that $s>0$ and $\varphi^n(s)>0$, for any $n$. Sequence $\varphi^n(s)$ is monotone decreasing
because $\varphi(r) < r$ and $0$ is one its lower bound. So it is convergent sequence. 
Let $\lim_{n \rightarrow +\infty} \varphi^n(s) = b > 0$. 
Then there exists positive integer $n_0$ such that $n\ge n_0$ implies $\varphi^n(s)\le b$,
because $$\overline{\lim_{t\to b+}}\varphi (t)< b.$$ So we obtain the contradiction, because $\varphi^n(s)>\varphi^{n+1}(s)>b$ for any $n$. 

Let $\beta$) be satisfied. Then $\phi^2$ satisfies $\alpha$).

\end{proof}

\begin{Lemma}\label{L2.2} Let $\varphi:[0,+\infty) \rightarrow [0,+\infty)$ be comparison function, which satisfies  
$$\overline{\lim_{t\rightarrow r-}}\varphi (t)< r  \mbox{ and } \overline{\lim_{t\rightarrow r+}}\varphi (t)\le r$$ for each $r>0$, such that
for every $s>0$ which satisfies
$$\overline{\lim_{t\rightarrow s+}}\varphi(t)=s,$$ there exists $\varepsilon > 0$ such that 
$$\varphi (t) = s\enskip \text{ for any } t\in (s,s+\varepsilon).$$
Then $\phi:[0,+\infty) \rightarrow [0,+\infty)$ defined by $\phi(0)=0$ and
$$\phi(t)=\sup_{t\in [0,\eta]}\varphi(t)$$
is monotone non - decreasing comparison function such that $\DS{\lim_{n\rightarrow+\infty}\varphi^n(r) = 0}$, which satisfies $$\varphi(t)\le\phi(t),$$ for any $t> 0$.\end{Lemma}

\section{The $d-CS$ spaces}
Here we define $d-CS$ spaces and give some properties of those spaces.
\begin{Definition}\label{D3.1} Let $X$ be arbitrary set, $d: X\times X\rightarrow [0,+\infty)$ and $\tau_{d}$ topology on $X$  defining by the family of closed sets as follows: a set $A\subseteq X$ is closed if and only if for each $x\in X$, $d(A,x)=0$ implies $x\in A$, where
$$d(A,x)=\inf\{d(a,x):a\in A\}.$$
Then, ordered triplet $(X,d,\tau_{d})$ is $d-C$ space.
\end{Definition}

Let $(X,d,\tau_{d})$ be $d-C$ space and $x\in X$. By $B(x,r)=\{y\in X:~d(y,x)<r\}$, we denote ball with center $x$ and radius $r$. Such ball
need not be open set. 
\begin{Example}Let $X=[0, 1]$ and $d:X\times X\rightarrow[0, +\infty)$ defined by $d(x, y)=1-|x-y|$, for all $x, y\in X.$ Then $A=\{1\}$ is not closed set because, 
$$D(A, 0)=0 \mbox{ and } 0\notin A.$$ Also, $B(1, 1)=\{y\in X : 1-|1-y|<1\}=[0, 1)$ is not open set.
\end{Example}
\begin{Proposition}\label{P3.1} If $(X,d,\tau_{d})$ is a $d-C$ space, then the family $\{B(x,r):r>0\}$ forms a local basis at $x$. Also, if $d(x_n,x)\rightarrow 0$ then $x_n\rightarrow x$ in the topology $\tau_d$.
\end{Proposition}

\begin{proof}Let $U\ni x$ be an open set. Then $X\setminus U$ is closed and, since $x\notin X\setminus U$, we
have $d(X\setminus U,x)=\eta>0$. Hence $B(x,\eta)\subseteq U$.

Assume that $d(x_n,x)\rightarrow0$. If $U$ is an open set that contains $x$, then $U\supseteq B(x,\eta)$ for
some $\eta>0$. The last set contains almost all members of the sequence $x_n$.\end{proof}

The convergence of a sequence $(x_n)$ in the topology $\tau_{d}$
need not imply $d(x_n,x)\to 0$.

The fact that $\lim d(x_n,x)=0$ will be denoted by $\lim^d x_n=x$.

\begin{Definition}\label{D3.2} Let $(X,d,\tau_{d})$ be a $d-C$ space. A function $f:X\rightarrow X$ is said to be left sequentially $d$-continuous
 if from $\lim d(x_n,x)=0$ follows $\lim d(f(x_n),f(x))=0$, for any $x\in X$ and each $(x_n)\subseteq X$. \end{Definition}

\begin{Definition}\label{D3.3}
Let $(X,d,\tau_{d})$ be a $d-C$ space. We define the following two properties:

\item{\rm(W3)} $\lim d(x_n,x)=0$ and $\lim d(x_n,y)=0$ implies  $x=y$;

\item{\rm(JMS)} $\lim d(x_n,z_n)=0$ and $\lim d(y_n,z_n)=0$ implies $\overline{\lim} d(x_n,y_n)\ne +\infty$.
\end{Definition}

The properties (W3) was introduced by M.\ Fr\'{e}chet \cite{MF1} and (JMS) by J. Jachy\-mski, J.~Matkowski and T.~Swiatkowski \cite{JMS}. 

Next statement gives the characterization of symmetric space which satisfies the property (JMS). It generalize famous result of J.~Jachymski, J.~Matkowski and T.~Swiatkowski \cite{JMS} obtained for symmetric spaces.

\begin{Theorem}\label{T3.1}
Let $(X,d)$ be a $d-C$ space. Then the following
conditions are equivalent:\\
(i) $(X,d)$ satisfies property (JMS);\\
(ii) There exists $\delta,\eta>0$ such that for any $x,y,z\in X$,
$$d(x,z)+d(y,z)<\delta \mbox{ implies that }d(x,y)<\eta.$$
(iii) There exists $r>0$ such that $$R=\sup{\{\di (B(x,r)):x\in
X\}}<+\infty.$$
\end{Theorem}

\begin{proof} $(i)\Rightarrow (iii)$. Suppose that  $\sup{\{\di (B(x,r)):x\in
X\}}=+\infty$ for any $r>0$ and $\lim^dx_n=z$. Then for each $n$ there exists
$y_n$ such that $d(y_n,z)<d(x_n,z)$ and $d(x_n,y_n)>n$. So we obtain that $\neg (iii)\Rightarrow \neg (i)$.

\noindent $(iii)\Rightarrow (ii)$. Put $\DS{ \delta = \frac r2}$ and $\eta = R$.

\noindent $(ii)\Rightarrow (i)$. From $\lim d(x_n,z_n)=0$ and $\lim d(y_n,z_n)=0$ follows
that there exists natural number $N$ such that $n>N$ implies 
$\DS{d(x_n,z_n)<\frac {\delta}2}$ and $\DS{d(y_n,z_n)<\frac {\delta}2}$. Hence,
$\DS{d(x_n,y_n)<\eta}$ for any $n>N$. 
\end{proof}
\begin{Remark} (ii) is equivalent with: 
\begin{equation}
d(x,z)+d(y,z)<\delta \mbox{ implies } \max\{d(x,y),d(y,x)\}<\eta.
\end{equation}
\end{Remark}
\begin{Definition}\label{D3.4} 
Let $(X,d,\tau_{d})$ be a $d-C$ space. A sequence $( x_{n})\subseteq X$ \ is said to be left Cauchy sequence if for given $\varepsilon >0$ there is
$N\in \mathbb{N}$ such that $d(x_{n},x_{m})<\varepsilon $, for all
$m>n\geq N.$ If $(X,d,\tau_{d})$ is $d-C$ space, then it is $d-CS$ space if and only if every
left Cauchy sequence converges to some $x\in X$. \end{Definition}

\begin{Lemma}\label{L3.1} Let $(X,d,\tau_{d})$ be a $d-CS$ space, which satisfies the properties (W3) and (JMS)$, f,g:X\rightarrow X$ be two
left sequentially $d$-continuous mappings
and  $d_*:X\times X\rightarrow [0,+\infty)$ be defined by : $ d_*(x, y) = 0 \mbox{ for }x=y$
and  
\begin{equation}
d_*(x,y)=\max\{d(x,y),\rho(f(x),g(y)),\ldots,\rho(f^n(x),g^n(y)\},
\end{equation}
otherwise. Then space $(X,d_*,\tau_{d_*})$  is left complete $d-C$ space.
\end{Lemma}

\begin{proof} The space $(X,d_*,\tau_{d_*})$ is $d-C$ metric space, which satisfies the properties (W3) and (JMS), because conditions of Definitions \ref{D3.1} and \ref{D3.3} are trivial satisfied.
Also, we have $d(x,y) \le d_*(x, y) $ for any $x,y\in X$.  Further, if $(x_j)\subseteq X$ an arbitrary left Cauchy sequence in $(X,d_*,\tau_{d_*})$, then $(x_j)$ is a left Cauchy sequence in $(X,d,\tau_{d})$, which implies that
 $(X,d_*,\tau_{d_*})$ is  $d-CS$ space, because  $(X,d,\tau_{d})$  is $d-CS$ space.  
\end{proof}

\section{Main Results}

By $\Phi$ we denote set of all comparison functions  $\varphi:[0,+\infty) \rightarrow [0,+\infty)$ which satisfies  
$$\overline{\lim_{t\rightarrow r-}}\varphi(t)< r  \mbox{ and } \overline{\lim_{t\rightarrow r+}}\varphi(t)\le r$$ for each $r>0$, such that
for every $s>0$ which satisfies
$$\overline{\lim}_{t\rightarrow s+}\varphi(t)=s,$$ there exists $\varepsilon > 0$ such that 
$$\varphi (t) = s\enskip \text{ for any } t\in (s,s+\varepsilon).$$
\begin{Example}Let $\varphi:[0,+\infty) \rightarrow [0,+\infty)$ defined by
$$\varphi(t)=\left\{\begin{array}{cc}
                      \frac{t}{2}, & t\in[0, \frac{1}{2}], \\
                      \frac{1}{2}, & t>\frac{1}{2},
                    \end{array}\right.$$
then $\varphi\in \Phi$.                    
\end{Example}
\begin{Lemma}\label{L4.1} Let $\varphi_1,\ldots,\varphi_n \in\Phi$. Then there exists monotone non-decreasing function
$\varphi\in\Phi$ such that $$\varphi_k (x)\le\varphi (x),$$ for
each $1\le k\le n$ and $x\ge 0$. 
\end{Lemma}

\begin{proof} Let $$\varphi (x) = \max\{\varphi_1(x),\ldots,\varphi_n(x)\}.$$
Then it is easy to show that: $\varphi (0)=0$, $\varphi_k(t)\le\varphi (t) < t$ ($1\le k\le n$) for
all $t>0$ and $$\overline{\lim_{t \rightarrow x+}} \varphi(t)< x.$$
\end{proof}

\begin{Lemma}\label{L4.2} Let $(X,d,\tau_{d})$ be a $d-CS$ space which
satisfies the properties (W3) and (JMS), let $f:X\rightarrow X$, let $\varphi\in \Phi$ and let $\delta,\eta$ be defined as in $(ii)$ of Theorem \ref{T3.1}. If
$$d(f(x),f(y))\le \varphi (d(x,y))$$
for any $x,y\in X$ and $$\sup_{t\in [0,\eta]}\varphi(t)\le\frac{\delta}2,$$ then $f$ has a unique fixed point $y\in X$ and for each
$x\in X$ the sequence of Picard iterates defined by $f$ at $x$ converges, in the topology $\tau_d$, to
$y$.\end{Lemma}

\begin{proof} For any $p,q\in X$ we have $$d(f(p),f(q))\le \varphi(d(p,q))\le
d(p,q),$$ which implies that $f$ is left sequentially $d$-continuous.

By Lemma \ref{L2.2} function $\phi:[0,+\infty) \rightarrow [0,+\infty)$ defined by $\phi(0)=0$ and
$$\phi(t)=\sup_{t\in [0,\eta]}\varphi(t)$$
is monotone non-decreasing comparison function such that $\DS{\lim_{n\rightarrow+\infty}\varphi^n(r) = 0}$, which satisfies $$\varphi(t)\le\phi(t),$$ for any $t> 0$ and $\varphi(t)\le\frac{\delta}2$.

Let $x\in X$. Then
$$d(f^n(x), f^{m+n}(x))\le \phi^n (d(x,f^m(x)))\mbox{ for any $m,n\in {\bf N}$}.$$
So
$$d(f^n(x),f^{n+1}(x))\le \phi^n (d(x,f(x))),$$
which implies that
$$d(f^n(x),f^{n+1}(x))\rightarrow 0.$$
Then there exists $k\in {\bf N}$ such that
$$d(f^k(x),f^{k+1}(x))\le\min\{\frac{\delta}2,\eta\}.$$

We shall prove that for all $n\in{\bf N}$,
\begin{equation}\label{eta}d(f^k(x),f^{k+n}(x))\le\eta.\end{equation}
By definition of $k$, we get that (\ref{eta}) is valid for $n=1$. Now, assume that (\ref{eta}) is satisfied
for some $n\in{\bf N}$. From
$$d(f^k(x),f^{k+1}(x))\le \frac {\delta}2$$
and
$$d(f^{k+1}(x),f^{k+n+1}(x))\le \phi(d(f^{k}(x),f^{k+n}(x))) \le\phi(\eta)\le \frac {\delta}2,$$
it follows that
$$d(f^k(x),f^{k+1}(x))+d(f^{k+1}(x),f^{k+n+1}(x))\le \delta ,$$
which by Theorem \ref{T3.1} implies that
$$d(f^{k}(x),f^{k+n+1}(x))\le\eta.$$
So, by induction we get that (\ref{eta}) is satisfied for any $n\ge 1$. Thus
$$d(f^{k+n}(x),f^{k+n+m}(x))\le \phi^n (\eta),\mbox{ for any $m,n\in {\bf N}$}.$$
Hence $(f^n(x))$ is a Cauchy sequence.

Then there exists $y\in X$ such that $\lim f^n(x)=y$. Since $f$ is left sequentially $d$-continuous, we have $\lim^d
f^{n+1}(x)=f(y)$. Now we get that $f(y)=y$ because $(X,d,\tau_{d})$ satisfies (W3).

Since $d(f^n(x),y)=0$, by Proposition \ref{P3.1} we have that $f^n(x)\rightarrow y$ in the topology $\tau_d$.

If $y^*$ is another fixed point for $f$, then for all $n$ we have
$$d(y,y^*)=d(f^n(y),f^n(y^*))\le\phi^n(d(y,y^*))\rightarrow0,\mbox{ as }n\rightarrow\infty.$$
\end{proof}

\begin{Theorem}\label{T4.1} Let $(X,d,\tau_{d})$ be a $d-CS$ space which
satisfies the properties (W3) and (JMS), $f:X\rightarrow X$ and $\varphi\in \Phi$. If
\begin{equation}
d(f(x),f(y))\le\varphi (d(x,y))
\end{equation}
for any $x,y\in X$, then $f$ has a unique fixed point $y\in X$ and for each $x\in X$ the sequence of Picard
iterates defined by $f$ at $x$ converges in $d$, and, also by Proposition \ref{P3.1}, it converges the same limit point in the topology $\tau_d$.
\end{Theorem}

\begin{proof} Let $\delta,\eta$ be defined as in (ii) of Theorem
\ref{T3.1}. By Lemma \ref{L2.2} function $\phi:[0,+\infty) \rightarrow [0,+\infty)$ defined by $\phi(0)=0$ and
$$\phi(t)=\sup_{t\in [0,\eta]}\varphi(t)$$
is monotone non-decreasing comparison function such that $\DS{\lim_{n\rightarrow+\infty}\varphi^n(r) = 0}$, 
which satisfies $$\varphi(t)\le\phi(t),$$ for any $t> 0$.

If $\phi(\eta )\le\delta/2$, then from Lemma \ref{L4.2}, it follows that $f$ has a unique fixed
point $y\in X$ and for each $x\in X$ the sequence of Picard iterates defined by $f$ at $x$ converges, in the
topology $\tau_d$, to $y$.

Now assume that ${\DS \phi(\eta )> \frac {\delta}2}$. Then there exists the least positive integer $j>1$
such that $\DS \phi^j(\eta )\le\delta/2$. Also, we have that
$$d(f^j(x),f^j(y))\le \phi^j(d(x,y)),$$
for any $x,y\in X$, $\phi^j\in \Phi$. By Lemma \ref{L4.2} we obtain that $f^j$ has a unique fixed point, say
$z\in X$ and for each $x\in X$ the sequence of Picard iterates defined by $f^j$ at $x$ converges, in the
topology $\tau_d$, to $z$. From Lemma \ref{L2.1} it follows that $f$ has a unique fixed point $z\in X$ and for
each $x\in X$ the sequence of Picard iterates defined by $f$ at $x$ converges, in the topology $\tau_d$, to $z$.
\end{proof}

\begin{Theorem}\label{T4.2} Let $(X,d,\tau_{d})$ be a $d-CS$ space which
satisfies the properties (W3) and (JMS), $f:X\rightarrow X$ be left sequentially $d$-continuous and $\varphi_1,\varphi_2,\varphi_3,\in \Phi$. If
\begin{equation}
d(f(x),f(y))\le\max\{\varphi_1(d(x,y)),\varphi_2(d(x,f(x)),\varphi_3d(y,f(y))\})
\end{equation}
for any $x,y\in X$, then $f$ has a unique fixed point $y\in X$ and for each $x\in X$ the sequence of Picard
iterates defined by $f$ at $x$ converges in $d$. Also, 
Picard iterates converges the same limit point in the topology $\tau_d$.
\end{Theorem}

\begin{proof}  By Lemma \ref{L4.1} there exists monotone non-decreasing
$\varphi\in\Phi$ such that $$\varphi_k (x)\le\varphi (x),$$ for
each $1\le k\le 3$ and $x\ge 0$. We get that
$$d(f(x),f(y))\le\varphi (\max\{d(x,y),\varphi d(x,f(x),\varphi d(y,f(y))\}).$$ 

Let $\delta,\eta$ be defined as in (ii) of Theorem
\ref{T3.1}. By Lemma \ref{L2.2} function $\phi:[0,+\infty) \rightarrow [0,+\infty)$ defined by $\phi(0)=0$ and
$$\phi(t)=\sup_{t\in [0,\eta]}\varphi(t)$$
is monotone non - decreasing comparison function such that $\DS{\lim_{n\rightarrow\infty}\varphi^n(r) = 0}$, which satisfies $$\varphi(t)\le\phi(t),$$ for any $t> 0$.

Define $d^*: X\times X\rightarrow [0,+\infty)$ by $d^*(x,y)=0$ if $x=y$ and
$$d^*(x,y)=\max\{d(x,y),d(x,f(x)),d(y,f(y))\}$$
otherwise. Then $(X,d^*,\tau_{d^*})$ is $d-C$ space. Also, we have 
$$d(x,y)\le d^*(x,y)$$
for any $x,y\in X$. 

So, if $(x_n)\subseteq X$ is an arbitrary left Cauchy sequence in $(X,d^*,\tau_{d^*})$ then it is left Cauchy sequence in $(X,d,\tau_{d})$, which implies its convergence. Hence, $(X,d^*,\tau_{d^*})$ is $d-CS$ space. 

From $\lim d^*(x_n,x)=0$ and $\lim d^*(x_n,y)=0$ it follows $\lim d(x_n,x)=0$ and $\lim d(x_n,y)=0$, which implies $x=y$, because $(X,d,\tau_{d})$ has property (W3). So, $(X,d^*,\tau_{d^*})$ has property (W3).

From $d^*(x,z)+d^*(y,z)<\delta$ we get following inequalities:
$d(x,z)+d(y,z)<\delta$, $d(x,f(x))+d(y,z)<\delta$ and $d(x,f(x))+d(y,f(y))<\delta$, which implies that $$d^*(x,y)\le \eta + 2\delta = \eta^*,$$
because $(X,d,\tau_{d})$ has property (JMS). So, $(X,d^*,\tau_{d^*})$ has property (JMS).

Let $x,y\in X$. From
\begin{equation*}
 d(f(x),f^2(x))\le \phi( d(x,f(x))), d(f(y),f^2(y))\le \phi( d(y,f(y)))
 \end{equation*}
 and
 \begin{equation*}
 d(f(x),f(y))\le \varphi d(x,y),
\end{equation*}
we obtain
$$d^*(f(x),f(y))\le\varphi (d^*(x,y)).$$

Now statement follows from Theorem \ref{T4.1}.
\end{proof}

The following theorem extends the previous results presented by M.\ Marjanovi\'c \cite{MM}.

\begin{Theorem} \label{T4.3}
Let $(X,d,\tau_{d})$ be a $d-CS$ space which
satisfies the properties (W3) and (JMS), $f:X\rightarrow X$ be left sequentially $d$-continuous and $\varphi\in\Phi$. If
\begin{equation}
d(f^{n+1}(x),f^{n+1}(y))\le\varphi(\max_{0\le i\le n}\{d(f^i(x),f^i(y))\}),
\end{equation}
for any $x,y\in X$, then $f$ has a unique fixed point $y\in X$ and for each $x\in X$ the sequence of Picard
iterates defined by $f$ at $x$ converges in $d$. Also, 
Picard iterates converges the same limit point in the topology $\tau_d$.
\end{Theorem}

\begin{proof} By Lemma \ref{L3.1} space $(X,d_*,\tau_{d_*})$ be a $d-CS$ space which satisfies the properties (W3) and (JMS).
Further, we get that
$$
\rho_*(f^{n+1}(x),f^{n+1}(y))\le\varphi(d_*(x,y)).
$$
By Theorem \ref{T4.1} it follows that  $f^{n+1}$  has unique  fixed point, say $q$ which is unique limit of all Picard sequences defined by $f^{n+1} $. 
By Lemma \ref{L2.1} we obtain that  $q$ is unique fixed point for $f$ and unique limit of all Picard sequences defined by $f$.
\end{proof}

\section{Fixed point result on quasi-metric space}
In this section we give a result for nonlinear contractions in quasi-metric spaces.

\begin{Theorem}\label{T5.1} Let $(X,d)$ be a quasi-metric space,
$f:X\rightarrow X$ be left sequentially $d$-continuous and $\psi_1,\psi_2,\psi_3\in \Phi$. If
$$d(f(x),f(y))\le\max\{\psi_1(d(x,y)),\psi_2(d(x,f(x))),\psi_3(d(y,f(y)))\}$$
for any $x,y\in X$, then $f$ has a unique fixed point $y\in X$ and for each $x\in X$ the sequence of Picard
iterates defined by $f$ at $x$ converges in $d$. Also, 
Picard iterates converges the same limit point in the topology $\tau_d$.
\end{Theorem}

\begin{proof}  Let $\psi:[0,+\infty) \rightarrow [0,+\infty)$ be a mapping defined by formula $$\psi(t)=\max\{\psi_1(t),\psi_2(t),\psi_3(t)\}.$$ Then $\psi(t)\in\Phi$ and
$$ d(f(x),f(y))\le \max\{\psi(d(g(x),g(y))),\psi(d(g(x),f(x))),\psi(d(g(y),f(y)))\}.$$

Let $x_{0}\in X$ be arbitrary and $(f(x_{n}))$ sequence of Picard iterates of $f$ at point $x_0$. 

Now we shall proved that  $\lim d(f(x_{n-1}),f(x_{n})) = 0$.

Let $d(f(x_{n-1}),f(x_{n})) = a_{n}$, $a_{1} = b_{1}$ and
$b_{n+1} = \psi(b_{n})$. From 
$$d(f(x),f^2(x))\le\psi(\max\{d(x,f(x)),d(x,f(x)),d(f(x),f^2(x)))\}$$
it follows
$$d(f(x),f^2(x))\le\psi(d(x,f(x))).$$
So
\begin{eqnarray}
a_n&=&d(f(x_{n-1}),f(x_{n})) \leq \psi(d(f(x_{n-2})) \leq \cdots\le\psi^{n-1}(d(f(x_{0}),f(x_{1})))\nonumber \\  
&=& \psi^{n-1}(b_{1})=b_n,\nonumber
\end{eqnarray}
it follows that $0 \leq a_{n} \leq b_{n}$, because $d$ is nonnegative mappings. So $\lim  b_{n} = \lim  a_{n} =0$, because $\lim b_n=0$.

Further we shall proved that  $\lim d(f(x_{n}),f(x_{n-1})) = 0$.

Let $d(f(x_{n}),f(x_{n-1})) = c_{n}$, $c_{1} = d_{1}$ and
$d_{n+1} = \psi(d_{n})$. From 
$$d(f^2(x),f(x))\le\psi(d(f(x),x))$$
%
%
%
we obtain
\begin{eqnarray*}
c_n&=&d(f(x_{n}),f(x_{n-1})) \leq 
\psi(d(f(x_{n-1}),f(x_{n-2}))) \leq \cdots\\ 
&\leq& \psi^{n-1}(d(f(x_{0}),f(x_{1})))= \psi^{n-1}(d_{1})=d_n,
\end{eqnarray*}
it follows that $0 \leq c_{n} \leq d_{n}$, because $d$ is nonnegative mappings. So $\lim  d_{n} = \lim  c_{n} =0$, because $\lim c_n=0$.

Now we shall proved that $(x_n)$ is a left Cauchy sequence.
Assume now that $(x_n)$ is not a Cauchy sequence. Then there exists
$\varepsilon >0$ and sequences of positive integers
$(m_{k})$ and $(n_{k})$, such that
for every $k \in N$:

\noindent i) $m_{k} > n_{k} \geq k$

\noindent ii) $ h_{k} = d(f(x_{n_{k}}),f(x_{m_{k}})) \geq \varepsilon$,
where $m_{k}$ is the smallest positive iteger which satisfies ii).

It follows that $d(f(x_{m_{k} -1}),f(x_{n_{k}})) < \varepsilon$ for any $k$.
So
\begin{eqnarray} h_{k}& =& d(f(x_{n_{k}}),f(x_{m_{k}}) \leq d(f(x_{n_{k}}),f(x_{n_{k}+1}))+d(f(x_{n_{k}+1}),f(x_{m_{k}}))  \nonumber \\ 
&\leq& a_{n_{k}}+\varepsilon \leq  a_{k} + \varepsilon, \nonumber
\end{eqnarray}
which implies that $\lim h_{n} = \varepsilon$. Hence
\begin{eqnarray*}
h_{k}&=& d(f(x_{n_{k}}),f(x_{m_{k}})) \\ 
&\leq& d(f(x_{n_{k}}),f(x_{n_{k}+1})) + d(f(x_{n_{k}+1}),f(x_{m_{k}+1})) \\
&+&d(f(x_{m_{k}+1}),f(x_{m_{k}}))\\ 
&\leq& a_{n_k+1}+\max\{\psi(h_k),\psi(a_{m_k+1}),\psi(c_{m_k+1})\}+ c_{m_k+1}\\
&\le& a_{k} + \psi(h_{k})+c_k.
\end{eqnarray*} 
We get that for some $k$ that
$\varepsilon \leq h_k\leq \varphi(h_k)$ which is contradiction.
So $x_{n}$ is a left Cauchy sequence. It is convergent because $(X,d)$ is complete. Let $y\in X$ be its limit. Then $\lim f(x_n)=y$, because $f$ is left sequentially $d$-continuous.

\noindent Let $y_0\ne x_0$ be arbitrary and $y_n$  be sequence of Picard iterates of $f$ at point $y_0$. Hence
$$d(f(x_n),f(y_n)) \le \max\{ \psi(d(x_n,y_n)),\psi(d(x_n,f(x_n))),\psi(d(y_n,f(y_n)))\}.$$
\noindent When $n\rightarrow\infty$, we get that $d(f(x_n),f(y_n))\rightarrow 0$, because $\lim d(x_n,f(x_n)=0$,  $\lim \psi^n(d(x_0,y_0)=0$ and $\lim d(y_n,f(y_n))=0$. Hence, all sequences of Picard iterations have same limit.

If there exists $p,q\in X$ such that $p=f(p)$ and $q=g(q)$ then 
$$ d(p,q) \le \max\{\psi(d(p,p)),\varphi(d(p,q)),
\varphi(d(q,q))\}=\varphi(d(p,q)).$$ So $d(p,q)=0$, which implies $y_1=y_2$.\end{proof}
\begin{Remark}Note that Theorem \ref{T5.1} is an extension of some results from Pasicki, Theorem 2.5. in \cite{LP1} and Theorem 3.1 in \cite{LP2}.
\end{Remark}
\noindent\textbf{Acknowledgments:}
First author supported by Serbian Ministry of Science, Technological development and innovation - Grant number 451-03-47/2023-01/ 200105  3.2.2023.\\
\noindent\textbf{Author's contributions}: All authors have read and agreed to the published version of the~manuscript. \\
\noindent\textbf{Funding}: This research received no external~funding.\\ 
\noindent\textbf{Conflicts of interest}: The authors declare no conflict of~interest. \\
\noindent\textbf{Availability of data and materials}:  Data sharing is not applicable to this article as no data set were
generated or analyzed during the current study.

\end{document}